\documentclass[10pt,a4paper]{amsart}

\usepackage[T1]{fontenc}
\usepackage[utf8]{inputenc} 
\usepackage[english]{babel}
\usepackage{amsmath, amssymb, amsthm, amsfonts}
\usepackage{graphicx}
\usepackage{hyperref}
\usepackage{color,xcolor}
\usepackage{listings}
\usepackage{url}
\usepackage{float}
\usepackage{indentfirst}
\usepackage{fullpage}
\usepackage{fancyvrb} 
\usepackage{comment}
\DeclareMathOperator{\lcm}{lcm}

\theoremstyle{plain}
\newtheorem{thm}{Theorem}
\newtheorem{defin}[thm]{Definition}
\newtheorem{cor}[thm]{Corollary}

\newtheorem{rem}[thm]{Remark}

\newtheorem{conj}[thm]{Conjecture}

\newtheorem{prob}[thm]{Problem}
\newtheorem{lem}[thm]{Lemma}

\title{On the Diophantine Equation Involving Elementary Symmetric Polynomials and the Decomposition of Unity}
\author{}
\date{2026 January}

\author[Kiss]{S\'andor Z. Kiss}
\email{kiss.sandor@ttk.bme.hu}
\address{Department of Algebra and Geometry, Institute of Mathemetics, Budapest University of Technology and Economics, M\H{u}egyetem rkp. 3., H-1111 Budapest, Hungary; \newline \hspace*{4mm} 
HUN-REN-BME Stochastics Research Group, M\H{u}egyetem rkp. 3., H-1111 Budapest, Hungary.
\newline \hspace*{4mm}
HUN-REN Alfr\'ed R\'enyi Institute of Mathematics, Re\'altanoda utca 13--15., H-1053 Budapest,  Hungary; \newline \hspace*{4mm} 
}

\author[S\'andor]{Csaba S\'andor}
\email{sandor.csaba@ttk.bme.hu}
\address{Department of Stochastics, Institute of Mathemetics, Budapest University of Technology and Economics, M\H{u}egyetem rkp. 3., H-1111 Budapest, Hungary; \newline \hspace*{4mm} 
Department of Computer Science and Information Theory, Budapest University of Technology and Economics, M\H{u}egyetem rkp. 3., H-1111 Budapest, Hungary; \newline \hspace*{4mm}
HUN-REN Alfr\'ed R\'enyi Institute of Mathematics, Re\'altanoda utca 13--15., H-1053 Budapest,  Hungary; \newline \hspace*{4mm} 
MTA--HUN-REN RI Lend\"ulet ``Momentum'' Arithmetic Combinatorics Research Group, Re\'altanoda utca 13--15., H-1053 Budapest,  Hungary.}

\author[Zakarczemny]{Maciej Zakarczemny}
\email{maciej.zakarczemny@pk.edu.pl}
\address{Department of Applied Mathematics, Faculty of Computer Science and Mathhematics, Cracow University of Technology, Warszawska 24, 31-155 Krak\'ow, Poland}

\begin{document}

\maketitle

\section{abstract}
\noindent 
We consider the equality of the values of the $n$th and $k$th elementary symmetric polynomials of $n$ not necessarily distinct positive integers. For $k < n$, we prove that this equation always has a solution, but only finitely many solutions. Furthermore, we consider the equality of the values of the $n$th and $(n-2)$th elementary symmetric polynomials of $n$ not necessarily distinct positive integers.
In particular, we show that the number of solutions of this equation tends to infinity if $n$ tends to infinity.  

{\it 2010 Mathematics Subject Classification:} 11D72, 11D45

{\it Keywords and phrases:} diophantine equations; symmetric polynomials

\section{Introduction}
A classical problem concerns perfect powers among binomial coefficients.
For integers $k, l, x \ge 2$ and $n \ge 2k$, consider the equation
$$
\binom{n}{k} = x^{l}.
$$
Obviously, there are infinitely many solutions when $k=l=2$, since substituting $m=2n-1$ reduces the equation to the Pell equation $m^{2}-8x^{2}=1$. For $k = 3$, $l = 2$, the above equation has the only solution $n = 50$, $x = 140$, see \cite{Gy}. It follows from the work of Erd\H{o}s \cite{Er}, K. Gy\H{o}ry \cite{Gy}, H. Darmon and Merel \cite{DM} that apart from the case $k = l = 2$, the above equation has the only solution $n = 50$, $x = 140$, $k = 3$, $l = 2$.
For $j \in \{1,2,\dots,n\}$, the elementary symmetric polynomials in $n$ variables $x_{1}, \dots ,x_{n}$ are denoted by 
$$
\sigma_{j}(x_{1}, \dots ,x_{n}) = \sum_{1\le i_{1} < \dots < i_{j}\le n}x_{i_{1}}\cdots x_{i_{j}}.
$$
Obviously, the special case when $k = l$, the above equation
$
\binom{n}{k} = x^{k}
$
can be written in the form $\sigma_{n-k}(\underbrace{x, \dots ,x}_{n}) = \sigma_{n}(\underbrace{x, \dots ,x}_{n})$. Similarly, if $l\mid k$, then the equation $\binom{n}{l}=x^k$ is equivalent to
$$
\sigma_{n-l}(\underbrace{x^{\frac{k}{l}},
\dots,x^{\frac{k}{l}}}_n)=
\sigma_n(\underbrace{x^{\frac{k}{l}},\dots,x^{\frac{k}{l}}}_n).
$$

This observation led us to consider the Diophantine equations of the form \begin{equation}\label{eqk}\sigma _k(x_1,\dots ,x_n)=\sigma_n(x_1,\dots ,x_n),\end{equation} $1\le x_1\le \dots \le x_n$, $x_i\in \mathbb{Z}$, where $1 \le k < n$, $k,n$ are integers.\\
Let us denote the number of solutions $(x_1,\dots,x_n)\in(\mathbb{Z}^+)^n$ by $f_k(n)$. We have the following estimations for $f_k(n)$.

\begin{thm}\label{generalbound}
For every $1\le k<n$ we have 
$$
1\le f_k(n)<2^{n2^{\binom{n}{k}}}.
$$
\end{thm}

Moreover, one can obtain upper estimations for $f_{k}(n)$ in terms of the following arithmetic function.
Let $f(n)$ denote the number of factorizations of the natural number $n$ into factors larger than $1$ where the order of the factors does not count. We put $f(1)=1.$
Let
$$M_{k,n}:=\max\{x_n\colon (x_1,\dots,x_n)\in(\mathbb Z^+)^n,\ 1\le x_1\le\cdots\le x_n,\ \sigma_k(x_1,\dots,x_n)=\sigma_n(x_1,\dots,x_n)\},$$
$$N_{k,n}^{(l)}:=\max\{x_1x_2\dots x_{n-l}\colon (x_1,\dots,x_n)\in(\mathbb Z^+)^n,\ 1\le x_1\le\cdots\le x_n,\ \sigma_k(x_1,\dots,x_n)=\sigma_n(x_1,\dots,x_n)\},$$ where $1\le k< n.$
Note that, by Theorem~\ref{generalbound}, the positive integers $N_{k,n}^{(l)}$ and $M_{k,n}$ exist.\\
We prove the following lemma. The known growth of the function $f$ (see \cite{CEP}) allows one to estimate the right–hand side in (2).

\begin{lem}\label{upperbound} 
\begin{enumerate} Let $1 \le k < n$, $k,n\in \mathbb{Z}$. The number of solutions satisfies the following estimates
\item $f_k(n)\le (M_{k,n})^n$, 
\item $\displaystyle \sum_{k< n\le x}f_k(n)\le \sum_{l\le \max_{k< n\le x}N_{k,n}^{(0)}}f(l)$,
\end{enumerate} 
\end{lem}

\begin{thm}\label{generalbound2}
If $\sigma_k(x_1,\dots ,x_n)=\sigma_n(x_1,\dots ,x_n)$, then $$\sigma_n(x_1,\dots ,x_n)\ge \binom{n}{k}^{\frac{n}{n-k}}.$$ The equality holds if and only if $k=n-1$, $x_1=\dots =x_n=n$ and $k=n-2$, $\binom{n}{2}$ is a square, $x_1=\dots =x_n=\sqrt{\binom{n}{2}}$.
\end{thm}

\subsection{The case \texorpdfstring{$k=1$}{k=1}}
The equation (1) can be written in the form 
$$x_1+\dots +x_n=x_1\cdots x_n,\,n\ge 2$$ which is called {\it Equal sum and product problem}, which has been extensively investigated by Canfield, Erd\H{o}s and Pomerance, Weingartner, Nyblom, Ecker, Kurlandchik and Nowicki, Sándor and Zakarczemny, among others \cite{CEP,RW,ZAMUC,ZCMB,SZ,Ny,Eck,KN}. 
Other examples are recent results connected to factorizations \cite{BFT,BMS,Tij,Matv,ACM,Sci,Sc}.
Clearly, $(x_1,x_2,\dots ,x_n)=(1,1,\dots ,1,2,n)$ is a solution, because $1+1+\dots +1+2+n=2n$.
To give a complete description about the solution set of this diophantine equation is a still unsolved problem \cite{Guy2004}, \cite{Sc}. Even for the number of solutions, the following conjecture of Schinzel \cite{Sci} still remains open.
\begin{conj}
    One has $\displaystyle \lim_{n\to \infty}f_1(n)=\infty$ .
\end{conj}

Kurlandchik and Nowicki \cite{KN} proved that
if $x_1+\dots +x_n = x_1\cdots x_n$, then
\begin{enumerate}
\item $x_n \le n$,
\item $x_1\cdots x_n \le 2n$.
\end{enumerate}

\begin{rem}
If $x_1+x_2+\dots +x_n=x_1x_2\dots x_n$, then $x_n\ge 2$. By the above result of Kurlandchik and Nowicki, one can get that $x_1x_2\dots x_{n}\le 2n$ (see also \cite[Theorem 4]{KN}), that is $N_{1,n}^{(0)}=2n$. 
\end{rem}
Furthermore, Oppenheim proved \cite{Op} that
$$
\sum_{n\le x}f(n)=(1+o(1))\frac{xe^{2\sqrt{\log x}}}{2\sqrt{\pi }(\log x)^{3/4}}.
$$
It follows that  
$$
\sum_{2\le n\le x}f_1(n)\le \sum_{l\le 2x}f(l)=(2+o(1))\frac{xe^{2\sqrt{\log x}}}{2\sqrt{\pi }(\log x)^{3/4}}
$$

Moreover, Weingartner \cite{RW} proved that
$$\sum_{2\le n\le x}f_1(n)=(1+o(1))\frac{xe^{2\sqrt{ \log x}}}{2\sqrt{\pi }(\log x)^{3/4}}.
$$
Let $\varepsilon > 0$. As a corollary, one can get $f_1(n)<n^{1+\varepsilon }$ if $n$ is large enough.
In the other direction S\' andor and Zakarczemny \cite{SZ} proved the existence of an infinite increasing sequence $n_i$ of positive integers with $f_1(n_i)>n_i^{1-\varepsilon}$.





\subsection{The case \texorpdfstring{$k=2$}{k=2}} In this case the equation \ref{eqk} can be written in the form
 $$\sum_{1\le i<j\le n}x_ix_j=x_1\cdots x_n,$$ 
 $n \ge 3$.
 It is easy to check that a solution is $$(x_1,x_2,\dots ,x_n)=(1,1,\dots ,1,2,n,\frac{1}{2}n(3n-5)).$$

P. Miska and M. Ulas \cite{MU} proved that 
if $\sum_{1\le i<j\le n}x_ix_j=x_1\cdots x_n$, then
\begin{enumerate}
\item $x_n\le \frac{1}{2}n(3n-5)$
\item $x_1\cdots x_n\le n^2(3n-5)$,
\end{enumerate}
which implies that $N_{2,n}^{(0)}=n^{2}(3n-5)$. Then by Lemma \ref{upperbound} (2) and Oppenheim's bound we have $$\sum_{n\le x}f_2(n)\le \sum_{n\le x^2(3x-5)}f(n)=x^{3+o(1)},$$ which gives the upper estimation $$f_2(n)<n^{3+o(1)}.$$
\subsection{The case \texorpdfstring{$k=n-1$}{k=n-1}}
The equation (\ref{eqk}) is equivalent to the classical Egyptian fraction equation $$\sum_{i=1}^n\frac{1}{x_i}=1,$$
deeply connected with Sylvester's sequence \cite{ACMJ,OEISA000058}.
Let us define the sequence $u_n$ by $u_1=2$ and $u_{n+1}=u_n^2-u_n+1$ for $n\ge 1$, $(u_n)_{n=1}^{\infty} =(2,3,7,43,1807,\dots )$. Then $$u_{n+1}=u_1u_2\dots u_n+1.$$ By induction on $n$ we get $$\sum_{i=1}^{n-1}\frac{1}{u_i}+\frac{1}{u_n-1}=1,$$ that is
$$
\sigma_{n-1}(u_1,u_2,\dots ,u_{n-1},u_n-1)=\sigma_n(u_1,u_2,\dots ,u_{n-1},u_n-1).
$$
In 1950, Erd\H{o}s \cite{E} proved that
if $x_i\in\mathbb{Z^+}$ and $\sum_{i=1}^n\frac{1}{x_i}=1$, then $\max x_n \le u_n-1$. 
In 2005, Soundararajan \cite{S} generalized this result by proving that
if $x_i\in\mathbb{Z^+}$ and $\sum_{i=1}^n\frac{1}{x_i}<1$, then $$\sum_{i=1}^n\frac{1}{x_i}\le \sum_{i=1}^n\frac{1}{u_n}= 1-\frac{1}{u_{n+1}-1} = 1-\frac{1}{u_1u_2...u_n}.$$
Since $(u_{n-1}-1)^2<u_n<u_{n-1}^2$, the sequence $u_n^{1/2^n}$ is monotone decreasing, so the following limit exists $\lim_{n\to \infty }u_n^{1/2^n}=c_1=1.2640847353\ldots$. By Soundararajan's bound, it is easy to see that $$\sum_{i=1}^{n}\frac{1}{x_{i}} = 1,$$
$1 \le x_{1} \le x_{2} \le \dots \le x_{n}$, $x_{i}\in \mathbb{Z}$ implies that $x_{n-k} \le (k+1)u_{n-k}$ for $0 \le k \le n-1$. Hence $x_{1}x_{2}\dots x_{n} \le n!u_{1}u_{2}\dots u_{n}$.
It follows that $N_{n-1,n}^{(0)}=u_n^{2+o(1)}$. Hence (2) in Lemma \ref{upperbound} with $k=n-1$, $x=n$ and Oppenheim's bound give the estimation $f_{n-1}(n)\le u_n^{2+o(1)}$.

In 2011, Browning and Elsholtz \cite{BE} improved on this result by proving  

$f_{n-1}(n)\le u_n^{0.4+o(1)}=c_1^{(0.4+o(1))2^n}$
as $n\rightarrow \infty$.
On the other hand, in 2014 Konyagin \cite{K} gave the following double exponential lower bound.
There exists a $c>0$ such that
$f_{n-1}(n)>2^{2^{c\frac{n}{\log n}}}$.

\subsection{The case \texorpdfstring{$k=n-2$}{k=n-2}}
The equation (\ref{eqk}) is equivalent to \begin{equation}\label{sum1/xy=1}\sum_{1\le i<j\le n}\frac{1}{x_ix_j}=1.\end{equation}
\begin{rem}
For small $n$ we explicitly list all solutions of \eqref{sum1/xy=1}. For $n=2$ the unique solution is $(1,1)$. For $n=3$ the unique solution is $(1,2,3)$. For $n=4$ there are exactly two solutions: $(1,2,4,14)$ and $(2,2,2,6)$. For $n=5$ there are exactly $27$ solutions, namely 
$$
(1,2,4,15,{218}), (1,2,4,16,116), (1,2,4,17,82), (1,2,4,18,65), (1,2,4,20,48), (1,2,4,{26},31),
(1,2,5,9,163),$$
$$ (1,2,5,10,60), (1,2,6,7,152), (1,2,6,8,43), (1,2,7,7,35), (1,2,7,8,22), (1,2,7,9,17),
(1,2,{8},8,16),$$
$$ (1,3,3,8,129), (1,3,3,9,48), (1,3,3,12,21), (1,3,4,5,107), (1,{4},4,4,28),
(2,2,2,7,46), (2,2,2,8,26),$$ 
$$ (2,2,2,10,16), (2,2,2,11,14), (2,2,3,4,19), (2,2,4,4,8),
(2,3,3,3,9), ({3},3,3,3,4).
$$

\end{rem}
\begin{defin}\label{df1}
Let the sequence \( v_n \) be defined recursively by
$$
v_1 = 1, \quad v_2 = 2, \quad \text{and for } n \ge 2,\quad
v_{n+1} = 1 + \frac{\sum\limits_{1\le i\le n} \frac{1}{v_i} } { 1 - \sum\limits_{1 \le i < j \le n} \frac{1}{v_i v_j} }.
$$
\end{defin}
We have $(v_n)_{1}^\infty=(1,2,4,15,219,47863,2290845187,\ldots).$\\

Then
$$
\sum_{1\le i<j\le n-1}\frac{1}{v_iv_j}+\frac{1}{v_n-1}\sum_{i=1}^{n-1}\frac{1}{v_i}=1,
$$
that is
$$\sigma_{n-2}(v_1,v_2,\dots ,v_{n-1},v_n-1)=\sigma_n(v_1,v_2,\dots ,v_{n-1},v_n-1).$$
We will prove the following recursions.

\begin{thm}\label{18}
\begin{enumerate} For every $n\ge 2$,
\item $v_{n+1} = v_n^2+v_nv_{n-1}+v_{n-1}^2-v_{n-1}^3-v_n-v_{n-1}+1.
$
\item $v_{n+1}=v_n^2-v_n+1+\prod\limits_{1\le i\le n-1}v_i$
\item $v_{n+1}=1+\left(\prod\limits_{1\le i\le n}v_i\right)\left(\sum\limits_{1\le i\le n}\frac{1}{v_i}\right).$
\end{enumerate}
\end{thm}

Note that the sequence $v_n$ is an increasing sequence of positive integers.

It follows from (3) of Theorem \ref{18} that $v_{n+1}-1>v_1v_2\dots v_n$ and for $n>2$ that $$(v_n-1)^2<v_{n+1}\le v_n^2.$$ 
Consequently, the sequence $v_n^{1/2^n}$ is monotone decreasing, and hence the limit
$$\lim_{n\to \infty }v_n^{1/2^n}=c_2
= 1.183382020799312\ldots$$ exists.  
We prove the following estimations.
\begin{thm}\label{upp}
If $b_i\in\mathbb{Z^+}$ and $\sum_{1\le i<j\le n}\frac{1}{b_ib_j}<1$, then $$\sum_{1\le i<j\le n}\frac{1}{b_ib_j}\le \sum_{1\le i<j\le n}\frac{1}{v_iv_j}= 1-\frac{1}{v_1v_2\cdots v_n}$$ and $$\sum_{i=1}^n\frac{1}{b_i}\le \sum_{i=1}^n\frac{1}{v_i}.$$
\end{thm}


\begin{thm}\label{xnvn-1}
If $\sigma_{n-2}(x_1,x_2,\dots ,x_n)=\sigma_n(x_1,x_2,\dots ,x_n)$, where $x_i\in\mathbb{Z^+}$ and $0<x_1\le x_2\le \dots \le x_n$, then $x_n\le v_n-1$.
\end{thm}

\begin{thm}\label{upper}\label{xn-k}
If $\sigma_{n-2}(x_1,x_2,\dots ,x_n)=\sigma_n(x_1,x_2,\dots ,x_n)$, where $x_i\in\mathbb{Z^+}$ and $0<x_1\le x_2\le \dots \le x_n$, then $x_{n-k}\le 2(k+1)v_{n-k}$.
\end{thm}


\begin{cor}\label{prodc}
If $\sigma_{n-2}(x_1,x_2,\dots ,x_n)=\sigma_n(x_1,x_2,\dots ,x_n)$, where $x_i\in\mathbb{Z^+}$ and $0<x_1\le x_2\le \dots \le x_n$, then 
\begin{enumerate}
\item $ {\binom{n}{2}}^\frac{n}{2}\le x_1x_2\cdot\ldots\cdot x_n\le 2^{n-1}n!v_1v_2\cdot\ldots\cdot v_n$,
\item $x_1x_2 \cdot\ldots\cdot x_{n-2}\le 2^{n-2}(n-1)!v_1v_2\cdot\ldots\cdot v_{n-2}$.
\end{enumerate}
\end{cor}

It follows that $N_{n-2,n}^{(0)}=v_n^{2+o(1)}$. Hence (2) in Lemma \ref{upperbound} with $k=n-2$, $x=n$ and Oppenheim's bound give the estimation $f_{n-2}(n)\le v_n^{2+o(1)}$.

\begin{lem}\label{Nn-2}
One has $\displaystyle f_{n-2}(n)\le (N_{n-2,n} ^{(2)})^{1+o(1)}$ as $n\to \infty $.    
\end{lem}

By Corollary \ref{prodc} (2) we have $x_1x_2\dots x_{n-2}=v_{n}^{0.5+o(1)}$ as $n\to \infty$. It follows that

\begin{cor}
One hs $f_{n-2}(n)<v_n^{0.5+o(1)}$.
\end{cor}

\begin{thm}\label{lower}
For $n\ge 4$, $f_{n-2}(n)\ge n-3$.
\end{thm}

\begin{prob}
Is it true that if $\sigma_{n-2}(x_1,x_2,\dots ,x_n)=\sigma_n(x_1,x_2,\dots ,x_n)$, where $0<x_1\le x_2\le \dots \le x_n$, then $x_1x_2\dots x_n\le v_1v_2\dots v_{n-1}(v_n-1)$?
\end{prob}

For a positive integer $m$, we denote
\[
1_{m} = \underbrace{1, \dots ,1}_{m}
\]


\section{Proofs}

\begin{proof}[Proof of Theorem \ref{generalbound}]
Lower bound: First we define the positive integer $p_k(n)$ for $1\le k<n$. Let $p_1(n)=n$. If we have defined the integers $p_i(n)$ for $i<k$ and $i<n$, then let
$$
p_k(n)=\sigma_k(1_{n-k-1},2,p_1(n-k+1)+1,p_2(n-k+2)+1,\dots ,p_{k-1}(n-1)+1).
$$
We prove by induction on $k$ that 
\begin{eqnarray*}
&&\sigma_k(1_{n-k-1},2,p_1(n-k+1)+1,p_2(n-k+2)+1,\dots ,p_{k-1}(n-1)+1,p_k(n))\\
=&&\sigma_n(1_{n-k-1},2,p_1(n-k+1)+1,p_2(n-k+2)+1,\dots ,p_{k-1}(n-1)+1,p_k(n))
\end{eqnarray*}
and
\begin{eqnarray*}
&&1=\sigma_n(1_{n-k-1},2,p_1(n-k+1)+1,p_2(n-k+2)+1,\dots ,p_{k-1}(n-1)+1,p_k(n)+1)\\
-&&\sigma_k(1_{n-k-1},2,p_1(n-k+1)+1,p_2(n-k+2)+1,\dots ,p_{k-1}(n-1)+1,p_k(n)+1)
\end{eqnarray*}
For $k=1$ we have
$$
\sigma_1(1_{n-2},2,n)=\sigma_n(1_{n-2},2,n)
$$
and
$$
1=\sigma_n(1_{n-2},2,n+1)-\sigma_1(1_{n-2},2,n+1).
$$
Let us suppose that 
\begin{eqnarray*}
&&\sigma_k(1_{n-k-1},2,p_1(n-k+1)+1,p_2(n-k+2)+1,\dots ,p_{k-1}(n-1)+1,p_k(n))\\
=&&\sigma_n(1_{n-k-1},2,p_1(n-k+1)+1,p_2(n-k+2)+1,\dots ,p_{k-1}(n-1)+1,p_k(n))
\end{eqnarray*}
and
\begin{eqnarray*}
&&1=\sigma_n(1_{n-k-1},2,p_1(n-k+1)+1,p_2(n-k+2)+1,\dots ,p_{k-1}(n-1)+1,p_k(n)+1)\\
-&&\sigma_k(1_{n-k-1},2,p_1(n-k+1)+1,p_2(n-k+2)+1,\dots ,p_{k-1}(n-1)+1,p_k(n)+1)
\end{eqnarray*}
for every $k<n$. Let $k+1<n$. We prove the equations 
\begin{eqnarray*}
&&\sigma_{k+1}(1_{n-(k+1)-1},2,p_1(n-(k+1)+1)+1,p_2(n-(k+1)+2)+1,\dots ,p_{(k+1)-1}(n-1)+1,p_{k+1}(n))\\
=&&\sigma_n(1_{n-(k+1)-1},2,p_1(n-(k+1)+1)+1,p_2(n-(k+1)+2)+1,\dots ,p_{(k+1)-1}(n-1)+1,p_{k+1}(n))
\end{eqnarray*}
and
\begin{eqnarray*}
&&\sigma_{n}(1_{n-(k+1)-1},2,p_1(n-(k+1)+1)+1,p_2(n-(k+1)+2)+1,\dots ,p_{(k+1)-1}(n-1)+1,p_{k+1}(n)+1)\\
-&&\sigma_{k+1}(1_{n-(k+1)-1},2,p_1(n-(k+1)+1)+1,p_2(n-(k+1)+2)+1,\dots ,p_{(k+1)-1}(n-1)+1,p_{k+1}(n)+1)\\ 
=&& 1.\\
\end{eqnarray*}
The equation
\begin{eqnarray*}
&&\sigma_{k+1}(1_{n-(k+1)-1},2,p_1(n-(k+1)+1)+1,p_2(n-(k+1)+2)+1,\dots ,p_{(k+1)-1}(n-1)+1,p_{k+1}(n))\\
=&&\sigma_n(1_{n-(k+1)-1},2,p_1(n-(k+1)+1)+1,p_2(n-(k+1)+2)+1,\dots ,p_{(k+1)-1}(n-1)+1,p_{k+1}(n))
\end{eqnarray*}
is equivalent to
\begin{eqnarray*}
&&\sigma_{k+1}(1_{n-(k+1)-1},2,p_1(n-(k+1)+1)+1,p_2(n-(k+1)+2)+1,\dots ,p_{(k+1)-1}(n-1)+1)\\
+&&p_{k+1}(n)\sigma_k(1_{n-(k+1)-1},2,p_1(n-(k+1)+1)+1,p_2(n-(k+1)+2)+1,\dots ,p_{(k+1)-1}(n-1)+1)\\
=&&p_{k+1}(n)\sigma_{n-1}(1_{n-(k+1)-1},2,p_1(n-(k+1)+1)+1,p_2(n-(k+1)+2)+1,\dots ,p_{(k+1)-1}(n-1)+1),
\end{eqnarray*}
which holds since
\begin{eqnarray*}
1=&&\sigma_{n-1}(1_{(n-1)-k-1},2,p_1((n-1)-k+1)+1,p_2((n-1)-k+2)+1,\dots ,p_k(n-1)+1)\\
-&&\sigma_k(1_{n-k-2},2,p_1(n-(k+1)+1)+1,p_2(n-(k+1)+2)+1,\dots ,p_k(n-1)+1)
\end{eqnarray*}
and
$$
p_{k+1}(n)=\sigma_{k+1}(1_{n-(k+1)-1},2,p_1(n-(k+1)+1)+1,p_2(n-(k+1)+2)+1,\dots ,p_{(k+1)-1}(n-1)+1).
$$
On the other hand, by induction,
\begin{eqnarray*}
&&\sigma_{n}(1_{n-(k+1)-1},2,p_1(n-(k+1)+1)+1,p_2(n-(k+1)+2)+1,\dots ,p_{(k+1)-1}(n-1)+1,p_{k+1}(n)+1)\\
-&&\sigma_{k+1}(1_{n-(k+1)-1},2,p_1(n-(k+1)+1)+1,p_2(n-(k+1)+2)+1,\dots ,p_{(k+1)-1}(n-1)+1,p_{k+1}(n)+1)\\
=&&\sigma_n(1_{n-(k+1)-1},2,p_1(n-(k+1)+1)+1,p_2(n-(k+1)+2)+1,\dots ,p_{(k+1)-1}(n-1)+1,p_{k+1}(n))\\
+&&\sigma_{n-1}(1_{n-(k+1)-1},2,p_1(n-(k+1)+1)+1,p_2(n-(k+1)+2)+1,\dots ,p_{(k+1)-1}(n-1)+1)\\
-&&\sigma_{k+1}(1_{n-(k+1)-1},2,p_1(n-(k+1)+1)+1,p_2(n-(k+1)+2)+1,\dots ,p_{(k+1)-1}(n-1)+1,p_{k+1}(n))\\
-&&\sigma_k(1_{n-(k+1)-1},2,p_1(n-(k+1)+1)+1,p_2(n-(k+1)+2)+1,\dots ,p_{(k+1)-1}(n-1)+1)\\
=&&\sigma_{n-1}(1_{(n-1)-k-1},2,p_1((n-1)-k+1)+1,p_2((n-1)-k+2)+1,\dots ,p_k(n-1)+1)\\
-&&\sigma_k(1_{(n-1)-k-1},2,p_1((n-1)-k+1)+1,p_2((n-1)-k+2)+1,\dots ,p_k(n-1)+1)\\
=&&1,
\end{eqnarray*}
which completes the induction.

Upper bound:
The equation (\ref{eqk}) is equivalent to
$$
\sum_{1\le i_1<i_2<\dots <i_{n-k}\le n}\frac{1}{x_{i_1}x_{i_2}\dots x_{i_{n-k}}}=1.
$$
It is well known \cite{E} that for the Diophantine equation $\sum_{i=1}^m\frac{1}{y_i}=1$ with $1\le y_1\le y_2\le \dots \le y_m$ we have $y_m < u_m<2^{2^m}$. It follows that $x_i\le x_{k+1}x_{k+2}\dots x_n\le u_{\binom{n}{k}}<2^{2^{\binom{n}{k}}}$, so the number of suitable $n$-tuples $(x_1,\dots ,x_n)$ is at most $2^{n2^{\binom{n}{k}}}$, which completes the proof.
 

\end{proof}


\begin{proof}[Proof of Lemma \ref{upperbound}]
(1): Each solution of the equation (\ref{eqk}) satisfies $x_i\in\{1,\dots,M_{k,n}\}$, hence there are at most $(M_{k,n})^n$ possible $n$-tuples. Therefore $f_k(n)\le (M_{k,n})^n$.\\
(2): Let $k<n\le x$. Suppose that for $1\le x_1\le \dots \le x_n$, $x_i\in \mathbb{Z}$ we have (\ref{eqk}). Then there exists an integer $k < t\le n$ such that such that $$1=x_1=\dots =x_{n-t}<x_{n-t+1}\le \dots \le x_n.$$ Then $$\sigma_{n}(x_1,x_2,\dots ,x_n)=x_1\cdots x_n=x_{n-t+1}\cdots x_n=\sigma_{k}(x_1,x_2,\dots ,x_n)\le N_{k,n}^{(0)} \le \max_{m\le x}N_{k,m}^{(0)}.$$
It is enough to show that for any $1<X_1\le \cdots \le X_s$, $X_i\in \mathbb{Z}$, $s > k$ with $X_1\cdots X_s\le \max_{n\le x}N_{k,n}^{(0)}$ there is at most one $n$-tuple $(x_{1}, \dots ,x_{n})$, $1 \le x_{1} \le x_{2} \le \dots \le x_{n}$ with $\sigma_{k}(x_{1}, \dots ,x_{n}) = \sigma_{n}(x_{1}, \dots ,x_{n})$, $x_1\cdots x_k\le \max_{n\le x}N_{k,n}^{(0)}$ and $(x_{1}, \dots ,x_{n}) = (1_{n-s}, X_{1}, \dots ,X_{s})$. For any nonnegative integer $u$, define the $u+s$-tuple $$(x_1^{(u+s)},\dots ,x_{u+s}^{(u+s)})=(1_{u},X_1,\dots ,X_s).$$ 
Then for each $u\ge 0$ we have
$$
\sigma_{u+s}(x_1^{(u+s)},\dots ,x_{u+s}^{(u+s)})=X_1\cdot\dots \cdot X_s.
$$
Hence
$$X_1\cdot\dots \cdot X_s=\sigma_{s}(x_1^{(s)},\dots ,x_{s}^{(s)}) 
=\sigma_{s+1}(x_1^{(s +1)},\dots ,x_{s+1}^{(s +1)}) 
=\sigma_{s+2}(x_1^{(s+2)},\dots ,x_{s+2}^{(s+2)}) 
=\dots $$ 
but
$$\sigma_k(x_1^{(s)},\dots ,x_{s}^{(s)}) < \sigma_k(x_1^{(s+1)},\dots ,x_{s+1}^{(s+1)}) < \sigma_k(x_1^{(s+2)},\dots ,x_{s+2}^{(s +2)}) <\dots$$ 
So there are at most one nonnegative integer $u\le x-s$ with $$\sigma_k(x_1^{(u+s)},\dots ,x_{u+s}^{(u+s)})=\sigma_{u+s}(x_1^{(u+s)},\dots ,x_{u+s}^{(u+s)}).$$
Hence
$$
\sum_{n\le x}f_k(n)\le \sum_{\ell\le \max_{n\le x} N^{(0)}_{k,n}} f(\ell).
$$
\end{proof}

\begin{proof}[Proof of Theorem \ref{generalbound2}]
Assume $x_1,\dots,x_n\in\mathbb{Z}^+$ and $k<n$. The equation $\sigma_k(x_1,\dots,x_n)=\sigma_n(x_1,\dots,x_n)$ is equivalent, after dividing by $x_1\cdots x_n$, to
$$
\sum_{1\le i_1<\cdots<i_{n-k}\le n}\frac1{x_{i_1}\cdots x_{i_{n-k}}}=1.
$$
By the geometric harmonic mean inequality, 
 $$\binom{n}{k} = \frac{\binom{n}{k}}{\sum\limits_{1 \le i_1 < i_2 < \dots < i_{n-k}\le n} \frac{1}{x_{i_{1}} x_{i_{2}} \dots x_{i_{n-k}}}} \le \sqrt[\binom{n}{k}]{(x_{1}\cdots x_{n})^{\binom{n-1}{k}}}.$$
 Therefore $$ \binom{n}{k}^{\tfrac{n}{n-k}} \le x_{1}\cdots x_{n}.$$ Equality holds if and only if $x_1=\cdots=x_n=x\in\mathbb{Z}^+$. Then $\sigma_k(x,x,\dots ,x)=\sigma_n(x,x,\dots ,x)$ is equivalent to $x^{\,n-k}=\binom nk$. Thus equality can occur if and only if $\binom nk$ is a perfect $(n-k)$th power. If $n-k=1$, then $x=n$. If $n-k=2$, then $\binom nk=\binom n2$ must be a square, with $x=\sqrt{\binom n2}$. If $n-k\ge3$, then $\binom nk$ would be a perfect power of exponent at least $3$, which is impossible by the theorems of Erdős \cite{Er} and K. Gy\H{o}ry \cite{Gy} on perfect powers among binomial coefficients, hence equality does not occur for $k\le n-3$.

\end{proof}

\begin{proof}[Proof of Theorem \ref{18}]
(1): Denote  
$$
S_n = \sum_{1 \le i \le n} \frac{1}{v_i}, \quad Q_n = \sum_{1 \le i < j \le n} \frac{1}{v_i v_j}.
$$
Then  
$$
S_n = S_{n-1} + \frac{1}{v_n}, \quad Q_n = Q_{n-1} + \frac{1}{v_n} S_{n-1}
.$$
By definition of the sequence $v_n$
$$
v_n = 1 + \frac{S_{n-1}}{1 - Q_{n-1}}
.$$
Hence  
$$
1 - Q_{n-1} = \frac{S_{n-1}}{v_n - 1}
.$$
It follows that  
$$
v_{n+1} = 1 + \frac{S_{n-1} + \frac{1}{v_n}}{1 - \left(Q_{n-1} + \frac{1}{v_n} S_{n-1}\right)} 
= 1 + \frac{S_{n-1} + \frac{1}{v_n}}{S_{n-1} \left( \frac{1}{v_{n}-1} - \frac{1}{v_n} \right)} 
= 1 + v_n (v_{n} - 1) + \frac{1}{S_{n-1}} (v_n - 1)
.$$
Therefore
\begin{equation}\label{we1}
S_{n-1} (v_{n+1}  - v_n (v_{n} - 1)-1) = v_n - 1 
\end{equation}
Hence
$$
S_{n-2} (v_n - v_{n-1}(v_{n-1} - 1) - 1) = v_{n-1} - 1, \mathrm{where\,\,} S_{n-2} = S_{n-1}-\frac{1}{v_{n-1}}.
$$
Therefore  
\begin{equation}\label{we2}
S_{n-1} (v_n - v_{n-1}(v_{n-1} - 1) - 1) = \frac{v_{n} - 1}{v_{n-1}}. 
\end{equation}
From \eqref{we1} and \eqref{we2} we get  
$$
(v_n - 1)(v_n - v_{n-1}(v_{n-1} - 1) - 1) = \frac{v_{n} - 1}{v_{n-1}}(v_{n+1} - v_{n}(v_n - 1) - 1).
$$

Finally  
$$
v_{n+1} = v_n^2 + v_n v_{n-1} + v_{n-1}^2 - v_{n-1}^3 - v_n - v_{n-1} + 1.
$$
(2): We have by Theorem \ref{18} (1)
$$
v_{n+1} -(v_n^2 - v_n+1)=v_{n-1}(v_n-(v_{n-1}^2 - v_{n-1}+1)). 
$$
Using this relation, we can prove the lemma by induction.\\
(3): By Theorem \ref{18} (2) we have
$$v_{n+2}-(v_{n+1}^2-v_{n+1}+1) =\prod\limits_{1\le i\le n}v_i.$$
Thus by \eqref{we1} we finally get
$$S_{n}\prod\limits_{1\le i\le n}v_i=v_{n+1}-1.$$
\end{proof}





\begin{proof}[Proof of Theorem \ref{upp}]
To prove Theorem \ref{upp} we need the following lemmas.

\begin{lem}
Let $x_1\ge x_2\ge \cdots \ge x_n>0$ and $y_i\ge 0$ be real numbers with $x_n\ge y_n$, $x_n+x_{n-1}\ge y_n+y_{n-1}$,\dots,$x_n+x_{n-1}+\dots +x_2\ge y_n+y_{n-1}+\dots +y_2$ and $\displaystyle x_n+x_{n-1}+\dots +x_1=y_n+y_{n-1}+\dots +y_1$. Then we have $\displaystyle \sum_{1\le i<j\le n}x_ix_j\ge \sum_{1\le i<j\le n}y_iy_j$ and equality holds if and only if $y_i=x_i$ for every $1\le i\le n$.
\end{lem}

Since
$$(x_1+\dots +x_n)^2=\sum_{i=1}^nx_i^2+2\sum_{1\le i<j\le n}x_ix_j =(y_1+\dots +y_n)^2=\sum_{i=1}^ny_i^2+2\sum_{1\le i<j\le n}y_iy_j,$$
we get the following equivalent statement.

\begin{lem}
Let $x_1\ge x_2\ge \cdots \ge x_n>0$ and $y_i\ge 0$ be real numbers with $x_n\ge y_n$, $x_n+x_{n-1}\ge y_n+y_{n-1},\dots ,x_n+x_{n-1}+\dots +x_2\ge y_n+y_{n-1}+\dots +y_2$ and $\displaystyle x_n+x_{n-1}+\dots +x_1=y_n+y_{n-1}+\dots +y_1$. Then we have $\displaystyle \sum_{i=1}^nx_i^2\le \sum_{i=1}^ny_i^2$ and equaility holds if and only if $y_i=x_i$ for every $1\le i\le n$.
\end{lem}

\begin{proof}
For given $x_{1} \ge x_{2} \ge \dots \ge x_{n} > 0$, define the set $S$ by
\begin{eqnarray*}
S=\{ (y_1,\dots,y_n): &&y_i\ge 0, y_n\le x_n, y_n+y_{n-1}\le x_n+x_{n-1},\dots ,y_n+y_{n-1}+\dots +y_2\\
&&\le x_n+x_{n-1}+\dots +x_2, y_n+y_{n-1}+\dots +y_1=x_n+x_{n-1}+\dots +x_1 \}.
\end{eqnarray*}

Let us suppose that the minimum of $\displaystyle \sum_{i=1}^ny_i^2$ is taken at $(u_1,\dots ,u_n)$. If $(u_1,\dots ,u_n)\ne (x_1,\dots ,x_n)$, then there exists an $1<i\le n$ such that $u_i+\dots +u_n<x_i+\dots +x_n$, where $i$ is the largest integer with this property. Then $u_i<x_i$. If $u_{j} \le x_{j}$ for every $1\le j<i$, then by $u_i+\dots +u_n<x_i+\dots +x_n$, we have $u_n+u_{n-1}+\dots +u_1\neq x_n+x_{n-1}+\dots +x_1$.
It follows that there exists a maximal $1\le j<i$ such that $u_j>x_j$ and so $u_i<x_i\le x_j<u_j$. If $$0< \delta <min\{ u_j-u_i,x_i+\dots +x_n-(u_i+\dots +u_n)\} ,$$ then $(u_1,\dots ,u_{j-1},u_j-\delta ,u_{j+1},\dots ,u_{i-1},u_i+\delta ,u_{i+1},\dots ,u_n)\in S$ and 
$$u_1^{2} + \dots + u_{j-1}^{2} + (u_j-\delta)^{2} + u_{j+1}^{2} + \dots +u_{i-1}^{2} + (u_i + \delta)^{2} + u_{i+1}^{2} + \dots + u_n^{2} < \sum_{i=1}^{n}u_{i}^{2}
$$
a contradiction.
\end{proof}

As a direct consequence we have

\begin{cor}
Let $x_1\ge x_2\ge \dots \ge x_n>0$ and $y_i\ge 0$ be real numbers with $x_n\ge y_n$, $x_n+x_{n-1}\ge y_n+y_{n-1}$,\dots ,$x_n+x_{n-1}+\dots +x_1\ge y_n+y_{n-1}+\dots +y_1$. Then we have $\displaystyle \sum_{1\le i<j\le n}x_ix_j\ge \sum_{1\le i<j\le n}y_iy_j$ and equality holds if and only if $y_i=x_i$ for every $1\le i\le n$.
\end{cor}

Finally, we need the following lemma.

\begin{lem}
Let $x_1\ge x_2\ge \dots \ge x_n>0$ and  $y_1\ge y_2\ge \dots \ge y_n>0$ be two decreasing sequence of positive real numbers. Suppose that $x_1x_2\dots x_j\ge y_1y_2\dots y_j$ for every $1\le j\le n$. Then
\begin{eqnarray}\label{xi}
\sum _{i=1}^nx_i\ge \sum _{i=1}^ny_i
\end{eqnarray}
and
\begin{eqnarray}\label{xixj}
\sum_{1\le i<j\le n}x_ix_j\ge \sum_{1\le i<j\le n}y_iy_j.
\end{eqnarray}
\end{lem}

\begin{proof} We follow Soundararajan's \cite{S} method. Set $x_{n+1}=\min (x_n,y_n)\frac{y_1y_2\dots y_n}{x_1x_2\dots x_n}$ and $y_{n+1}=\min (x_n,y_n)$. Then $x_{n+1}\le x_n$, $y_{n+1}\le y_n$ and $x_1x_2\dots x_{n+1}=y_1y_2\dots y_{n+1}$. By scaling we may also assume that $x_{n+1}\ge 1$ so that all the variables are at least 1. We now deduce our Lemma from Muirhead's theorem (see \cite{HLP} Theorem  45, pages 44--48). 

Let $\alpha_{1}, \dots ,\alpha_{n}$ be nonnegative real numbers and let $a_{1}, \dots ,a_{n} > 0$. Consider the sum
\[
\sum_{\sigma}a_{1}^{\alpha_{1}}a_{2}^{\alpha_{2}}\cdots a_{n}^{\alpha_{n}},
\]
where $\sigma$ runs over the permutations of $\{1, \dots n\}$. We say the sequence $(\alpha_{i})$ majorize the sequence $(\alpha_{i}^{'})$ if
\begin{itemize}
    \item [(i)] $\alpha_{1}+ \dots +\alpha_{n} = \alpha_{1}^{'} + \dots + \alpha_{n}^{'}$,
    \item [(ii)] $\alpha_{1} \ge \dots \ge \alpha_{n}$, $\alpha_{1}^{'} \ge \dots \ge \alpha_{n}^{'}$,
    \item [(iii)] $\alpha_{1} + \dots + \alpha_{\nu} \ge \alpha_{1}^{'} + \dots + \alpha_{\nu}^{'}$, for every $1 \le \nu \le n$.
\end{itemize}
The theorem of Muirhead asserts that if the sequence $(\alpha_{i})$ majorize the sequence $(\alpha_{i}^{'})$, then for all nonnegative $a_{i}$'s,
\[
\sum_{\sigma}a_{1}^{\alpha_{1}}a_{2}^{\alpha_{2}}\cdots a_{n}^{\alpha_{n}} \ge \sum_{\sigma}a_{1}^{\alpha_{1}^{'}}a_{2}^{\alpha_{2}^{'}}\cdots a_{n}^{\alpha_{n}^{'}}
\]
and there is equality only when $(\alpha_{i})$ and $(\alpha_{i}^{'})$ are identical
or when all the $a_{i}$'s are equal.

Now choose $\alpha _i=\log x_i$ and $\alpha _i'=\log y_i$ for $1\le i\le n+1$. To prove (\ref{xi}) take $a_1 = e$ and $a_2 =\dots = a_{n+1} = 1$. In order to verify (\ref{xixj}) take $a_1=a_2=e$ and $a_3 =\dots = a_{n+1} = 1$. The
hypotheses of Lemma 20 then give the hypotheses of Muirhead's theorem ((i)-(iii))
and the conclusion of Muirhead's theorem gives our desired inequalities.
\end{proof}

Now we prove Theorem \ref{upp} by induction on $k$.\\ For $k=2$ we have $\frac{1}{x_1x_2}\le \frac{1}{2}=\frac{1}{v_1v_2}$ and
$\frac{1}{x_1}+\frac{1}{x_2}\le \frac{3}{2}=\frac{1}{v_1}+\frac{1}{v_2}.$\\
By the definition of $v_n$ and Theorem \ref{18} (3) we have
$$
v_{k+1}=1+\frac{\sum\limits_{1\le i\le k} \frac{1}{v_i} } { 1 - \sum\limits_{1 \le i < j \le k} \frac{1}{v_i v_j} }=1+v_1v_2\dots v_k\sum\limits_{1\le i\le k}\frac{1}{v_i}.
$$
Hence  $$\sum_{1\le i<j\le k } \frac{1}{v_iv_j}=1-\frac{1}{v_1v_2\dots v_k}.$$

Let us suppose that $b_1b_2\dots b_k\ge  v_1v_2\dots v_k$. Let $1\le l\le k$ denote the largest integer $j$ such that $$b_jb_{j+1}\dots b_k\ge v_jv_{j+1}\dots v_k.$$  Then
\begin{eqnarray}\label{prod}
b_l\ge v_l,\quad b_lb_{l+1}\ge v_lv_{l+1},\quad b_lb_{l+1}b_{l+2}\ge v_lv_{l+1}v_{l+2},\dots
\end{eqnarray}
We shall prove that
\begin{eqnarray}\label{1bi}
\sum_{i=l}^k\frac{1}{b_i}\le \sum_{i=l}^k\frac{1}{v_i}
\end{eqnarray}
and
\begin{eqnarray}\label{1bibj}
\sum_{l\le i<j\le k}\frac{1}{b_ib_j}\le \sum_{l\le i<j\le k}\frac{1}{v_iv_j}.
\end{eqnarray}

Taking $x_1=\frac{1}{v_l}$, $x_2=\frac{1}{v_{l+1}}$, \dots ,$x_{k-l+1}=\frac{1}{v_k}$ and $y_1=\frac{1}{b_l}$, $y_2=\frac{1}{b_{l+1}}$,\dots ,$y_{k-l+1}=\frac{1}{b_k}$ in Lemma 20 we see that (\ref{prod}) implies (\ref{1bi}) and (\ref{1bibj}).

By induction, we get that
$$
\sum_{i=1}^k\frac{1}{b_i}=\sum_{i=1}^{l-1}\frac{1}{b_i}+\sum_{i=l}^k\frac{1}{b_i}\le \sum_{i=1}^{l-1}\frac{1}{v_i}+\sum_{i=l}^k\frac{1}{v_i}=\sum_{i=1}^k\frac{1}{v_i}
$$
and
\begin{eqnarray}
&&\sum_{1\le i<j\le k}\frac{1}{b_ib_j}= \sum_{1\le i<j\le l-1 } \frac{1}{b_ib_j} + \left( \sum_{i=1}^{l-1}\frac{1}{b_i}  \right)\left( \sum_{i=l}^k\frac{1}{b_i} \right) +  \sum_{l\le i<j\le k } \frac{1}{b_ib_j}\\ &&\le  \sum_{1\le i<j\le l-1 } \frac{1}{v_iv_j} + \left( \sum_{i=1}^{l-1}\frac{1}{v_i}  \right)\left( \sum_{i=l}^k\frac{1}{v_i} \right) +  \sum_{l\le i<j\le k } \frac{1}{v_iv_j} = \sum_{1\le i<j\le k } \frac{1}{v_iv_j}.
\end{eqnarray}


If  $b_1b_2\cdots b_k\le v_1v_2\cdots v_k$, then $$\sum_{1\le i<j\le k } \frac{1}{b_ib_j}\le  1-\frac{1}{b_1b_2\cdots b_k}\le  1-\frac{1}{v_1v_2\cdots v_k}= \sum_{1\le i<j\le k } \frac{1}{v_iv_j}.$$
We have to prove $\displaystyle \sum_{i=1}^k \frac{1}{b_i}\le  \sum_{i=1}^k \frac{1}{v_i}$. If there exists an $1\le l\le k$ such that $\displaystyle \sum_{i=l}^k\frac{1}{b_i}\le \sum_{i=l}^k\frac{1}{v_i}$, then by induction we have $$\sum_{i=1}^{l-1}\frac{1}{b_i} \le  \sum_{i=1}^{l-1}\frac{1}{v_i}$$ and  $$\sum_{i=l}^k \frac{1}{b_i} \le  \sum_{i=l}^k \frac{1}{v_i}.$$ It follows that $\displaystyle \sum_{i=1}^k \frac{1}{b_i}\le  \sum_{i=1}^k \frac{1}{v_i}$ and we are done.
So we may assume that $\displaystyle \sum_{i=l}^k\frac{1}{b_i}> \sum_{i=l}^k\frac{1}{v_i}$ for every $1\le l\le k$.

Upon writing $x_i=\frac{1}{b_i}$ and $y_i=\frac{1}{v_i}$ in Corollary 19, we get that $\displaystyle \sum_{1\le i<j\le k}\frac{1}{b_ib_j}>\sum_{1\le i<j\le k}\frac{1}{v_iv_j}$. We have already proved that if $\displaystyle \sum_{1\le i<j\le k}\frac{1}{b_ib_j}<1$, then $\displaystyle \sum_{1\le i<j\le k}\frac{1}{b_ib_j}\le \sum_{1\le i<j\le k}\frac{1}{v_iv_j}$. It follows that $\displaystyle \sum_{1\le i<j\le k}\frac{1}{b_ib_j}\ge 1$, a contradiction.

\end{proof}

\begin{proof}[Proof of Theorem \ref{xnvn-1}]
Let $0<b_1\le b_2\le \dots \le b_n$ such that $\sigma_{n-2}(b_1,b_2,\dots ,b_n)=\sigma_n(b_1,b_2,\dots ,b_n)$. Then $\displaystyle \sum_{1\le i<j\le n-1}\frac{1}{b_ib_j}<1$. If $\displaystyle \sum_{1\le i<j\le n-1}\frac{1}{b_ib_j}=1-\frac{A}{b_1b_2\dots b_{n-1}}$, then by Theorem\ref{18} (3) and Theorem \ref{upp}, $$b_n=\frac{b_1b_2\dots b_{n-1}}{A}\sum_{i=1}^{n-1}\frac{1}{b_i}\le v_1v_2\dots v_{n-1}\sum_{i=1}^{n-1}\frac{1}{v_i}=v_n-1.$$

\end{proof}

\begin{proof}[Proof of Therorem \ref{xn-k}]
By Theorem \ref{upp}, we have
\begin{eqnarray*}
&&\frac{1}{v_{n-k}}\le \frac{1}{v_{n-k}}\sum_{i=1}^{n-k-1}\frac{1}{v_i}+\frac{1}{v_{n-k+1}}\sum_{i=1}^{n-k}\frac{1}{v_i}+\dots +\frac{1}{v_{n-1}}\sum_{i=1}^{n-2}\frac{1}{v_i}+\frac{1}{v_n-1}\sum_{i=1}^{n-1}\frac{1}{v_i} \\
&=& 1-\sum_{1\le i < j < n-k}\frac{1}{v_{i}v_{j}} \le 1-\sum_{1\le i < j < n-k}\frac{1}{x_{i}x_{j}}\\
&= & \frac{1}{x_{n-k}}\sum_{i=1}^{n-k-1}\frac{1}{x_i}+\frac{1}{x_{n-k+1}}\sum_{i=1}^{n-k}\frac{1}{x_i}+\dots +\frac{1}{x_{n-1}}\sum_{i=1}^{n-2}\frac{1}{x_i}+\frac{1}{x_n}\sum_{i=1}^{n-1}\frac{1}{x_i}\\
&\le & \frac{1}{x_{n-k}}\sum_{i=1}^n\frac{1}{x_i}+\frac{1}{x_{n-k}}\sum_{i=1}^n\frac{1}{x_i}+\dots +\frac{1}{x_{n-k}}\sum_{i=1}^n\frac{1}{x_i}+\frac{1}{x_{n-k}}\sum_{i=1}^n\frac{1}{x_i}\\
&\le & \frac{1}{x_{n-k}}\sum_{i=1}^{\infty}\frac{1}{v_i}+\frac{1}{x_{n-k}}\sum_{i=1}^{\infty}\frac{1}{v_i}+\dots +\frac{1}{x_{n-k}}\sum_{i=1}^{\infty}\frac{1}{v_i}+\frac{1}{x_{n-k}}\sum_{i=1}^{\infty}\frac{1}{v_i}\\
&\le &\frac{2(k+1)}{x_{n-k}}.
\end{eqnarray*}
Hence
$$x_{n-k}\le 2(k+1)v_{n-k},$$ which completes the proof.
\end{proof}





\begin{proof}[Proof of Lemma \ref{Nn-2}]


For $k = n-2$, the equation (\ref{eqk}) is equivalent to 
\begin{eqnarray*}\label{startlemma}&&\sigma_{n-2}(x_1,x_2,\dots ,x_{n-2})+\sigma_{n-3}(x_1,x_2,\dots ,x_{n-2})x_{n-1}+\sigma_{n-3}(x_1,x_2,\dots ,x_{n-2})x_n+\\&&\sigma_{n-4}(x_1,x_2,\dots ,x_{n-2})x_{n-1}x_n=\sigma_{n-2}(x_1,x_2,\dots ,x_{n-2})x_{n-1}x_n.\end{eqnarray*}
We fix, for the moment, $x_1,\ldots,x_{n-2}$ and introduce the abbreviations
$$
A=\sigma_{n-2}(x_1,\ldots,x_{n-2}),\qquad
C=\sigma_{n-3}(x_1,\ldots,x_{n-2}),\qquad
B=\sigma_{n-4}(x_1,\ldots,x_{n-2}).
$$
Collecting terms in the variables $x_{n-1},x_n$, the above equality becomes
\begin{equation}\label{ABC}
A+C x_{n-1}+C x_n+B x_{n-1}x_n=A x_{n-1}x_n.
\end{equation}
Hence
$$
(A-B)x_{n-1}x_n=A+C(x_{n-1}+x_n).
$$
The right hand side is positive, thus $A-B>0$.
Next, equality \eqref{ABC} can be written as
$$
A(A-B)+C^2=\big((A-B)x_{n-1}-C\big)\big((A-B)x_n-C\big).
$$
This is the key point: for fixed $(x_1,\ldots,x_{n-2})$, the number of pairs $(x_{n-1},x_n)$ satisfying the equation can be bounded using the divisor function $d(m)$, where
$
m=A(A-B)+C^2,
$
since each pair corresponds to a~factorization $m=uv$ with
$
u=(A-B)x_{n-1}-C,\,\, v=(A-B)x_n-C.
$
It follows that
\begin{align*}
f_{n-2}(n)
&\le 2\sum_{x_1x_2\cdots x_{n-2}\le N_{n-2,n}^{(2)}}
d\Big(
\sigma_{n-2}(x_1,x_2,\ldots,x_{n-2})\big(\sigma_{n-2}(x_1,x_2,\ldots,x_{n-2})-\sigma_{n-4}(x_1,x_2,\ldots,x_{n-2})\big) \\
&\hspace{35mm}+\sigma_{n-3}(x_1,x_2,\ldots,x_{n-2})^2
\Big).
\end{align*}
If $
\sigma_{n-2}(x_1,x_2,\dots ,x_n)=\sigma_n(x_1,\dots ,x_n),$
then $\sum_{1\le i<j\le n}\frac{1}{x_ix_j}=1$. Hence $x_{n-2}\ge \sqrt{\binom{n-2}{2}}$, otherwise $\sum_{1\le i<j\le n-2}\frac{1}{x_ix_j}\ge \binom{n-2}{2}\frac{1}{x_{n-2}^2}>1$, a contradiction. It follows that
$$
\sqrt{\binom{n-2}{2}}\le x_1\cdot\ldots\cdot x_{n-2}=\sigma_{n-2}(x_1,\dots ,x_{n-2}).
$$
Note that
$$\sigma_{n-3}(x_1,x_2,\dots ,x_{n-2})\le (n-2)\sigma_{n-2}(x_1,x_2,\dots ,x_{n-2}).$$
If $n\ge 5$, then $ n-2\le \binom{n-2}{2} $, which implies that 
$$
\sigma_{n-3}(x_1,\dots ,x_{n-2})\le \sigma_{n-2}(x_1,\dots ,x_{n-2})^3.
$$
It is well known that $d(m)=m^{o(1)}$ as $m\to \infty$. It follows that  
\begin{eqnarray*}
&&d\Big(\sigma_{n-2}(x_1,x_2,\dots ,x_{n-2})(\sigma_{n-2}(x_1,x_2,\dots ,x_{n-2})-\sigma_{n-4}(x_1,x_2,\dots ,x_{n-2}))+\sigma_{n-3}(x_1,x_2,\dots ,x_{n-2})^2\Big)\\
=&& \Big(\sigma_{n-2}(x_1,x_2,\dots ,x_{n-2})(\sigma_{n-2}(x_1,x_2,\dots ,x_{n-2})-\sigma_{n-4}(x_1,x_2,\dots ,x_{n-2}))+\sigma_{n-3}(x_1,x_2,\dots ,x_{n-2})^2\Big)^{o(1)}\\
\le && (\sigma_{n-2}(x_1,x_2,\dots ,x_{n-2})^2+\sigma_{n-2}(x_1,x_2,\dots ,x_{n-2})^6)^{o(1)}
\le  \sigma_{n-2}(x_1,x_2,\dots ,x_{n-2})^{o(1)}\le ({N_{n-2,n}^{(2)}})^{o(1)}.
\end{eqnarray*}
By Oppenheim's Theorem, see \cite{Op},
$$
f_{n-2}(n)\le \sum_{x_1x_2\dots x_{n-2}\le N_{n-2,n}^{(2)}}(N_{n-2,n}^{(2)})^{o(1)}=(N_{n-2,n}^{(2)})^{1+o(1)}, 
$$
which completes the proof.
\end{proof}

The proof Theorem \ref{lower} is based on the following Lemma

\begin{lem}\label{lemmalower} Let $x_1\le x_2\le \dots \le x_{n-1}$ be positive integers with $$\sum_{1\le i<j\le n-1}\frac{1}{x_ix_j}=1-\frac{1}{\lcm_{1\le i<j\le n-1} (x_ix_j)}.$$ For $1\le i\le n-1$, let $x_n=\lcm_{1\le i<j\le n-1}(x_ix_j)\sum_{i=1}^{n-1}\frac{1}{x_i}$ and $x_i^*=x_i$.
\begin{enumerate}
\item Then we have $\sum_{1\le i<j\le n}\frac{1}{x_ix_j}=1$.
\item If $x_n^*=x_n+1$, then $$\sum_{1\le i<j\le n}\frac{1}{x_i^*x_j^*}=1-\frac{1}{\lcm_{1\le i<j\le n} (x_i^*x_j^*)}.$$
\item 
If the positive integers $x_i$, $1\le i\le n-1$, are even, define $x_i^{**}=x_i$ for $1\le i\le n-1$ and $x_n^{**}=x_n+2$. Then
$$
\sum_{1\le i<j\le n}\frac{1}{x_i^{**}x_j^{**}}=1-\frac{1}{\mathrm{lcm}_{1\le i<j\le n}(x_i^{**}x_j^{**})}.
$$
\item For $n\ge 3$, there exist n-tuples $(x_1^{(h,n)},x_2^{(h,n)},\dots ,x_n^{(h,n)})$, $1\le h\le n-2$ with $$x_1^{(h,n)}=x_2^{(h,n)}=x_3^{(h,n)}=2$$ and $$\sum_{1\le i<j\le n}\frac{1}{x_i^{(h,n)}x_j^{(h,n)}}=1-\frac{1}{\lcm_{1\le i<j\le n} (x_i^{(h,n)}x_j^{(h,n)})}.$$
\end{enumerate}
\end{lem}

\begin{proof} Let
$
L=\mathrm{lcm}_{1\le i<j\le n-1}(x_i x_j),\, L^*=\mathrm{lcm}_{1\le i<j\le n}(x_i^* x_j^*).
$
We have
$
\sum_{1\le i<j\le n-1}\frac{1}{x_i x_j}=1-\frac{1}{L}.
$ 
Note that $x_i\mid L$, hence
$
x_n=L\sum_{i=1}^{n-1}\frac{1}{x_i}=\sum_{i=1}^{n-1}\frac{L}{x_i}\in\mathbb Z.
$\\
(1): Clearly, $\frac{1}{x_n}\sum_{i=1}^{n-1}\frac{1}{x_i}=\frac{1}{L},$ hence $$\sum_{1\le i<j\le n}\frac{1}{x_i x_j}=\sum_{1\le i<j\le n-1}\frac{1}{x_i x_j}+\sum_{i=1}^{n-1}\frac{1}{x_i x_n}=\left(1-\frac{1}{L}\right)+\frac{1}{x_n}\sum_{i=1}^{n-1}\frac{1}{x_i}=\left(1-\frac{1}{L}\right)+\frac{1}{L}=1.$$

(2): We have \begin{eqnarray*} 
    &&\sum_{1\le i<j\le n}\frac{1}{x_i^*x_j^*}=\sum_{1\le i<j\le n-1}\frac{1}{x_ix_j}+\frac{1}{x_n+1}\sum_{i=1}^{n-1}\frac{1}{x_i}\\
    &&=\sum_{1\le i<j\le n-1}\frac{1}{x_ix_j}+\left(\frac{1}{x_n}-\frac{1}{x_n(x_n+1)}\right)\sum_{i=1}^{n-1}\frac{1}{x_i}=\sum_{1\le i<j\le n}\frac{1}{x_i x_j}-\frac{1}{(x_n+1)L}=1-\frac{1}{(x_n+1)L}.
\end{eqnarray*}
Thus it suffices to show that $L^*=(x_n+1)L.$
\\
Since there exists $A<L^*$ we have
$$\frac{A}{L^*}=\sum_{1\le i<j\le n}\frac{1}{x_i^*x_j^*}=\frac{(x_n+1)L-1}{(x_n+1)L},$$
where $\gcd((x_n+1)L-1,(x_n+1)L)=1,$ we have $(x_n+1)L\mid L^*.$\\ Since for every $1\le i<j\le n$ we have $x_i^*x_j^*\mid (x_n+1)L,$ thus
$
L^* \mid (x_n+1)L,
$
which completes the proof.\\  
(3): Let
$
L^{**}=\mathrm{lcm}_{1\le i<j\le n}(x_i^{**} x_j^{**}).
$ Thus, by proceeding analogously to the above,

\begin{eqnarray*}
    &&\sum_{1\le i<j\le n}\frac{1}{x_i^{**}x_j^{**}}=\sum_{1\le i<j\le n-1}\frac{1}{x_ix_j}+\frac{1}{x_n+2}\sum_{i=1}^{n-1}\frac{1}{x_i}\\
    &&=\sum_{1\le i<j\le n-1}\frac{1}{x_ix_j}+\left(\frac{1}{x_n}-\frac{2}{x_n(x_n+2)}\right)\sum_{i=1}^{n-1}\frac{1}{x_i}=1-\frac{2}{(x_n+2)L}.
\end{eqnarray*}
Now $x_n$ is even. It suffices to prove that $$L^{**}=\left(\frac{x_n}{2}+1\right)L.$$
Since there exists $A<L^{**}$ we have
$$\frac{A}{L^{**}}=\sum_{1\le i<j\le n}\frac{1}{x_i^{**}x_j^{**}}=\frac{\left(\frac{x_n}{2}+1\right)L-1}{\left(\frac{x_n}{2}+1\right)L},$$
where $\gcd\left(\left(\frac{x_n}{2}+1\right)L-1,\left(\frac{x_n}{2}+1\right)L\right)=1,$ we have $\left(\frac{x_n}{2}+1\right)L\mid L^{**}.$\\ On the other hand, for every $1\le i<j\le n-1$ we have $x_i^*x_j^*\mid \left(\tfrac{x_n}{2}+1\right)L$ and for every $1\le h\le n-1$ we have 
$$
x_h(x_n+2)\mid \left(\tfrac{x_n}{2}+1\right)L \Longleftrightarrow 2x_h\mid L,
$$
which is fulfilled since $x_1$ and $x_2$ are even.
It follows that for every $1\le i<j\le n$ we have
$
x_i^*x_j^*\mid \left(\tfrac{x_n}{2}+1\right)L.
$ Hence
$
L^{**}\mid \left(\tfrac{x_n}{2}+1\right)L,
$
which completes the proof.\\  
(4): We prove by induction on $n$ that there exist $n$-tuples $\left(x_1^{(h,n)},x_2^{(h,n)},\dots ,x_n^{(h,n)}\right)$, $1\le h\le n-2$ with $$x_1^{(h,n)}=x_2^{(h,n)}=x_3^{(h,n)}=2$$ and $$\sum_{1\le i<j\le n}\frac{1}{x_i^{(h,n)}x_j^{(h,n)}}=1-\frac{1}{\lcm_{1\le i<j\le n} \left(x_i^{(h,n)}x_j^{(h,n)}\right)},$$ where $x_{i}^{(n-2,n)}$, $1\le i\le n$ are even. 

If $n=3$, then $\left(x_1^{(1,3)},x_2^{(1,3)},x_3^{(1,3)}\right)=(2,2,2)$ is suitable. 

Suppose that for some $n\ge 3$ there exist $n$-tuples $\left(x_1^{(h,n)},x_2^{(h,n)},\dots ,x_n^{(h,n)}\right)$, $1\le h\le n-2$ with $$\sum_{1\le i<j\le n}\frac{1}{x_i^{(h,n)}x_j^{(h,n)}}=1-\frac{1}{\lcm_{1\le i<j\le n} \left(x_i^{(h,n)}x_j^{(h,n)}\right)}.$$ We may assume that $x_i^{(n-2,n)}$, $1\le i\le n$ are even. Let

\begin{eqnarray*}&&\left(x_1^{(h,n+1)},\dots ,x_n^{(h,n+1)},x_{n+1}^{(h,n+1)}\right)\\ =&&\left(x_1^{(h,n)},\dots ,x_n^{(h,n)},\lcm_{1\le i<j\le n}\left(x_i^{(h,n)}x_j^{(h,n)}\right)\sum_{i=1}^n\frac{1}{x_i}+1\right)\end{eqnarray*}
for $1\le h\le n-2$ and
\begin{eqnarray*}&&\left(x_1^{(n-1,n+1)},\dots ,x_n^{(n-1,n+1)},x_{n+1}^{(n-1,n+1)}\right)\\ =&&\left(x_1^{(n-2,n)},\dots ,x_n^{(n-2,n)},\lcm_{1\le i<j\le n}\left(x_ix_j\right)\sum_{i=1}^n\frac{1}{x_i}+2\right)\end{eqnarray*}

By Lemma \ref{lemmalower}  we have $$\sum_{1\le i<j\le n+1}\frac{1}{x_i^{(h,n+1)}x_j^{(h,n+1)}}=1-\frac{1}{\lcm_{1\le i<j\le n+1} \left(x_i^{(h,n+1)}x_j^{(h,n+1)}\right)}$$ for $1\le h\le n-1$ and $x_i^{(n-1,n+1)}$, $1\le i\le n+1$ are even, which completes the proof.
\end{proof}

\begin{proof}[Proof of Theorem \ref{lower}]
We get from Lemma \ref{lemmalower} (4) that there exist different $$\left(x_1^{(h,n-1)},x_2^{(h,n-1)},\dots ,x_{n-1}^{(h,n-1)}\right)$$ for $1\le h\le n-3$ with $$\sum_{1\le i<j\le n-1}\frac{1}{x_i^{(h,n-1)}x_j^{(h,n-1)}}=1-\frac{1}{\lcm_{1\le i<j\le n-1} \left(x_i^{(h,n-1)}x_j^{(h,n-1)}\right)}.$$ For $1\le h\le n-3$, let $$x_n^{(h,n-1)}=\lcm_{1\le i<j\le n-1}\left(x_i^{(h,n-1)}x_j^{(h,n-1)}\right)\sum_{i=1}^{n-1}\frac{1}{x_i^{(h,n-1)}}.$$ By Lemma \ref{lemmalower} (1) we have $$\sum_{1\le i<j\le n}\frac{1}{x_i^{(h,n-1)}x_j^{(h,n-1)}}=1,$$ which completes the proof.
\end{proof}

\section{Acknowledgements}
S\'andor Z. Kiss and Csaba S\'andor were supported by the NKFIH Grants No. K, K146387, KKP 144059.
S\'andor Z. Kiss was also supported by the National Research, Development and Innovation Office NKFIH (Excellence program, Grant Nr. 153829)

\end{document}